\documentclass[leqno]{amsart}

\usepackage{latexsym,amssymb}

\theoremstyle{plain} 

\newtheorem{theorem}{Theorem}[section]
\newtheorem*{Theorem B}{Theorem B}
\newtheorem*{Theorem A}{Theorem A}

\theoremstyle{remark} 
\newtheorem{remark}{Remark}[section]
\theoremstyle{remark}

\theoremstyle{definition}

\numberwithin{equation}{section}
\def\<{\left < }
\def\>{\right >}
\def\({\left ( }
\def\){\right )}

\def\e{\eqref}
\def\de{\delta}
\def\o{\omega}
\def\a{\alpha}
\def\b{\beta}
\def\p{\partial}
\def\sech{{\rm sech}\hskip.002in }

\begin{document}

\title[Ideal submanifolds with type number $\leq 2$] {Classification of ideal submanifolds of real space forms with type number $\leq 2$}

\author{Bang-Yen Chen and Handan Yildirim}

 \address{Department of Mathematics\\Michigan State University\\ East Lansing, MI 48824--1027, USA}

\email{bychen@math.msu.edu}

 \address{Department of Mathematics, Faculty of Science\\ Istanbul University\\ 34134 Vezneciler \\Istanbul, Turkey}
 \email{handanyildirim@istanbul.edu.tr}
\begin{abstract} Roughly speaking, an ideal immersion of a Riemannian manifold into a real space form is an isometric immersion which produces the least possible amount of tension from the ambient space at each point of the submanifold. The main purpose of this paper is to completely classify all non-minimal ideal submanifolds of real space forms with type number $\leq 2$.
 \end{abstract}



\maketitle

\section{Introduction.}  

Riemannian invariants play the most fundamental role in Riemannian geometry. They provide the intrinsic characteristics of Riemannian manifolds; moreover, they affect the behavior of  Riemannian manifolds in general.  Classically, among Riemannian curvature invariants people have studied sectional, Ricci and scalar  curvatures intensively since B. Riemann. 

Let $R^m(c)$ denote an $m$-dimensional real space form of constant sectional curvature $c$.
Given integers $n\geq 3$ and $k\geq 1$, let ${\mathcal S}(n,k)$ be the set  consisting of all unordered $k$-tuples $(n_1,\ldots,n_k)$ of  integers $\geq 2$ such that  $n_1< n$ and
$n_1+\cdots+n_k\leq n$. 
For each $(n_1,\ldots,n_k)\in {\mathcal S}(n,k)$, B.-Y. Chen introduced in \cite{c98,c00} a new type of curvature invariants, denoted by $\delta(n_1,\ldots,n_k)$ (see \cite{book} for details).

In \cite{c98,c00}, Chen proved that for every $n$-dimensional submanifold $M^n$ of $R^m(c)$ the invariant
$\delta(n_1,\ldots,n_k)$ and the squared mean curvature $H^2$ of $M^n$ satisfy the following optimal fundamental inequality:
\begin{align}\label{1.1}\delta{(n_1,\ldots,n_k)}\leq  c(n_1,\ldots,n_k)H^2+b(n_1,\ldots,n_k)
c,\end{align} where  $c(n_1,\ldots,n_k)$ and $b(n_1,\ldots,n_k)$
are positive constants defined by
\begin{align} &c(n_1,\ldots,n_k)= {{n^2(n+k-1-\sum n_j)}\over{2(n+k-\sum
n_j)}},\label{1.2}\\& b(n_1,\ldots,n_k)={{n(n-1)}\over2} -\sum_{j=1}^k
\frac{{n_j(n_j-1)}}{2}.\label{1.3}\end{align}

Many applications of the invariants $\delta(n_1,\ldots,n_k)$ and of the inequality \e{1.1} have been obtained during the last two decades (cf. \cite{c00,c08,book,c09} for details). For instance, it was shown that the $\delta$-invariants give rise to new obstructions to minimal, Lagrangian and slant immersions. It was also shown that these invariants relate closely with the first nonzero eigenvalue $\lambda_1$ of the Laplacian $\Delta$ on $M^n$. Moreover, they provide an optimal lower bound of  $\lambda_1$ for compact irreducible homogeneous spaces which improves a well-known result of T. Nagano \cite{N}. Furthermore, $\delta$-invariants have been applied in \cite{HV} by S. Haesen and L. Verstraelen to the theory of general relativity.

An isometric immersion of a Riemannian $n$-manifold $M^n$ into a real space form $R^m(c)$  is called {\it ideal\/} if it satisfies the equality case of inequality \e{1.1} identically for some $k$-tuple $(n_1,\ldots,n_k)$. Roughly speaking, an ideal immersion of a Riemannian manifold into a real space form is an isometric immersion which produces the least possible amount of tension from the ambient space at each point of the submanifold (see \cite{c00} or \cite[page 269]{book}). Since \e{1.1} is a very general and sharp inequality, it is a very natural and interesting problem to investigate  submanifolds which verify  the equality case of this inequality, i.e., to determine ideal immersions.

Recall that a hypersurface $M^{n}$ of a real space form $R^{n+1}(c)$ is said to have type number $\leq r$ if the shape operator  at each point $p\in M^{n}$ has at most $r$ nonzero eigenvalues. In general, a submanifold $M^{n}$ of $R^{m}(c)$ is said to have {\it type number} $\leq r$ if, at each point $p\in M^{n}$ and for each unit normal vector $\xi\in T^{\perp}_{p}M^{n}$, the shape operator $A_{\xi}$ has at most $r$ nonzero eigenvalues.

Since the invention of $\delta$-invariants in early 1990s,  $\delta$-invariants and the inequalities related to  these invariants have become a vibrant research subject in differential geometry. Many interesting results in this respect were obtained by many geometers (see, for instance, \cite{A,B,c93,c98,c00,c08,book,CM,Dillen1,Dillen2,DV,MO,S,V}). On the other hand, the classification of ideal submanifolds in space forms remains a very challenging problem. 
The purpose of this paper is thus to classify ideal submanifolds of real space forms with type number $\leq 2$.

\section{Preliminaries.}
Since $R^{m}(c)$ is of constant sectional curvature $c$, the Riemann curvature tensor $\tilde R$ of $R^m(c)$ satisfies
$$\tilde R(X,Y)Z=c\{\<Y,Z\>X-\<X,Z\>Y\}.$$
Assume that $\phi: M^{n}\to R^{m}(c)$ is an isometric immersion of an $n$-dimensional Riemannian manifold $M^n$ into $R^{m}(c)$. Denote by $\nabla$ and $\tilde\nabla$ the Levi-Civita connections on $M^n$ and $R^m(c)$, respectively. 

Let $X$ and $Y$ be vector fields tangent to $M^{n}$ and let $\zeta$ be a normal vector field of $M^{n}$. Then the formulas of Gauss and Weingarten give the following decomposition of the
vector fields $\tilde\nabla_XY$ and $\tilde\nabla_X\zeta$ into a tangent and a normal component:
\begin{align} &\label{2.1}\tilde \nabla_XY=\nabla_XY+h(X,Y), \\&
\tilde \nabla_X \zeta =-A_\zeta X+D_X\zeta.\end{align} These formulas define $h$, $A$ and $D$ which are called the second fundamental form, the shape operator and the normal connection, respectively.  For each $\xi$, $A_{\xi}$ is a symmetric endomorphism. 

The shape operator and the second fundamental form are related by
\begin{align}\label{2.3} \<h(X,Y),\xi\>=\<A_{\xi}X,Y\>.\end{align}
 The mean curvature vector field is defined by $H=\frac{1}{n}{\rm trace}\, h$.

The equations of Gauss, Codazzi and Ricci are given respectively by
\begin{align}\label{2.4} & \<R(X,Y)Z,W\> =  \<A_{h(Y,Z)} X,W\> - \<A_{h(X,Z)}Y,W\> \\ & \hskip.5in \notag+c\,(\<X,W\>\<Y,Z\>-\<X,Z\>\<Y,W\>),
\\ & \label{2.5} (\overline\nabla_X h)(Y,Z) = (\overline\nabla_Y h)(X,Z),\\
 & \label{2.5.1} \<R^{\perp}(X,Y)\xi,\eta\> = \<[A_{\xi},A_{\eta}]X,Y\>
\end{align}
for $X,Y,Z,W$ tangent to $M^{n}$ and $\xi,\eta$ normal  to $M^{n}$, where $R$ is the curvature tensor of $M^{n}$ and $\overline\nabla h$ is defined by
\begin{equation}\begin{aligned}\label{2.6}(\overline\nabla_X h)(Y,Z) = D_X h(Y,Z) - h(\nabla_X Y,Z) - h(Y,\nabla_X
Z).\end{aligned}\end{equation}

A submanifold $M^{n}$  is said to be {\it totally geodesic} if $h=0$ holds identically. It  is called {\it totally umbilical} if its second fundamental form satisfies
\begin{align}h(X,Y)=\<X,Y\>H.\end{align}
A totally umbilical submanifold is called an extrinsic sphere if its mean curvature vector field is a parallel normal vector field, i.e., $DH=0$ holds identically.

\section{$\delta$-invariants, inequalities and ideal immersions}

Let $M^n$ be a Riemannian  $n$-manifold. For a plane section $\pi\subset T_pM^n$, $p\in M^n$, let $K(\pi)$ be the sectional curvature of $M^n$ associated with $\pi$. For an orthonormal basis $\{e_1,\ldots,e_n\}$ of $T_pM^n$, the scalar curvature $\tau$ at $p$ is
defined by
\begin{align}\tau(p)=\sum_{i<j} K(e_i\wedge e_j). \end{align}

Let $L$ be a subspace of $T_pM^n$  with dimension $r\geq 2$  and let $\{e_1,\ldots,e_r\}$ be an orthonormal basis of $L$. The scalar curvature $\tau(L)$ of $L$ is defined by
\begin{align}\tau(L)=\sum_{\alpha<\beta} K(e_\alpha\wedge e_\beta),\quad 1\leq \alpha,\beta\leq r.\end{align}

As before, for given integers $n\geq 3$ and  $k\geq 1$, we denote by $\mathcal S(n,k)$ the finite set  consisting of all $k$-tuples $(n_1,\ldots,n_k)$ of integers  satisfying  $$2\leq n_1,\cdots,
n_k<n\;\; {\rm and}\;\; n_1+\cdots+n_k\leq n.$$ Moreover, we denote by ${\mathcal S}(n)$ the union $\cup_{k\geq 1}\mathcal S(n,k)$.

For each $(n_1,\ldots,n_k)\in \mathcal S(n)$ and each $p\in M^n$, the invariant $\de{(n_1,\ldots,n_k)}(p)$ is defined  by
\begin{align} \delta(n_1,\ldots,n_k)(p)=\tau(p)- \inf\{\tau(L_1)+\cdots+\tau(L_k)\},\end{align} where $L_1,\ldots,L_k$ run over all $k$ mutually orthogonal subspaces of $T_pM^n$ such that  $\dim L_j=n_j,\, j=1,\ldots,k$.

Chen proved in \cite{c98,c00} the following sharp general relation between $\delta{(n_1,\ldots,n_k)}$ and the squared mean curvature $H^2$ for  submanifolds in real space forms.

\begin{Theorem A} Let $M^n$ be an $n$-dimensional submanifold of a real space form $R^m(c)$ of constant  sectional curvature $c$. Then, for each $k$-tuple  $(n_1,\ldots,n_k)\in\mathcal S(n)$, we have
\begin{align}\label{3.4} \delta(n_1,\ldots,n_k) \leq  {{n^2(n+k-1-\sum n_j)}\over{2(n+k-\sum n_j)}}H^2 +{1\over2} \Big({{n(n-1)}}-\sum_{j=1}^k {{n_j(n_j-1)}}\Big)c.\end{align}

The equality case of inequality \eqref{3.4}  holds at a point $p\in M^n$ if and only if there exists an orthonormal basis  $\{e_1,\ldots,e_m\}$ at $p$ such that  the shape operators of $M^n$ in $R^m(c)$ at $p$  with respect to $\{e_1,\ldots,e_m\}$  take the following form:
\begin{align} \font\b=cmr10 scaled \magstep2
\def\bigzerol{\smash{\hbox{ 0}}}
\def\bigzerou{\smash{\lower.0ex\hbox{\b 0}}} A_r=\left[ \begin{matrix} A^r_{1} & \hdots & 0
\\ \vdots  & \ddots& \vdots &\bigzerou \\ 0 &\hdots &A^r_k&
\\ \\&\bigzerou & &\mu_rI \end{matrix} \right],\quad  r=n+1,\ldots,m,
\label{3.5}\end{align}
where $I$ is an identity matrix and $A^r_j$ is a symmetric $n_j\times n_j$  submatrix satisfying
\begin{align}\label{3.6}\hbox{\rm trace}\,(A^r_1)=\cdots=\hbox{\rm
trace}\,(A^r_k)=\mu_r.\end{align}\end{Theorem A}

An isometric immersion  of a Riemannian $n$-manifold $M^{n}$ into a real space form $R^{m}(c)$ is called {\it $\delta(n_1,\ldots,n_k)$-ideal\/} if it satisfies the equality case of inequality \e{3.4} identically. An isometric immersion of $M^{n}$ into $R^{m}(c)$ is called {\it ideal} if it is a $\delta(n_{1},\ldots,n_{k})$-ideal for some $(n_{1},\ldots,n_{k})\in {\mathcal S}(n)$.

 \section{Ideal submanifolds with type number $\leq 2$ in $\mathbb E^{m}$.}
 
In this section, we classify  ideal submanifolds of $\mathbb E^{m}$ with type number $\leq 2$.

\begin{theorem}\label{T:1} Let $M^{n}$ be an ideal submanifold of $\mathbb E^{m}$. If the type number of $M^{n}$ in $\mathbb E^{m}$ is  $\leq 2$, then either $M^{n}$ is a minimal submanifold or the immersion of $M^{n}$ in $\mathbb E^{m}$ is congruent to
\begin{equation}\begin{aligned} \label{4.1} &\text{\small$\Bigg($}\! \sqrt{1-a^{2}}\hskip.006in x_{1},x_{2},\ldots,x_{n-2},a x_{1}\sin x_{n-1}, ax_{1}\cos x_{n-1}\sin x_{n},\\& \hskip1.0in  a x_{1}\cos x_{n-1}\cos x_{n},0,\ldots,0 \text{\small$\Bigg)$}\end{aligned}\end{equation}
for some real number $a$ satisfying $0<a< 1$.
\end{theorem}
\begin{proof} First, let us assume that $M^{n}$ is a non-minimal, ideal hypersurface in $\mathbb E^{n+1}$. If  $M^{n}$ has type number $\leq 2$, then it follows from  \e{3.5}-\e{3.6} in Theorem A that $M^{n}$ is $(n_{1},n-n_{1})$-ideal for some integer $n_1\in [2,n-1]$. Thus, the shape operator $A_{e_{n+1}}$ of $M^{n}$ with respect to a unit normal vector $e_{n+1}$ has exactly one nonzero eigenvalue, say $\lambda$,  with multiplicity two. Therefore, there exists a local orthonormal frame $\{e_{1},\ldots,e_{n}\}$ of the tangent bundle of $M^{n}$ such that the second fundamental form satisfies 
\begin{equation}\begin{aligned}\label{4.2} &h(e_{n-1},e_{n-1})=h(e_{n},e_{n})=\lambda e_{n+1},\;\;  h(e_{i},e_{j})=0,\;\; otherwise.\end{aligned}\end{equation}

Let us put $\nabla_{X}e_{i}=\sum_{j=1}^{n}\o_{i}^{j}(X),\, i=1,\ldots,n.$
Then, after applying Codazzi's equation, we find
\begin{align}\label{4.4} &e_{\a}\lambda=\lambda \o_{n-1}^{\a}(e_{n-1})=\lambda \o_{n}^{\a}(e_{n}),\;\; e_{n-1}\lambda=e_{n}\lambda=0,\;\;  
\\&\label{4.5} \o_{n-1}^{\a}(e_{n})=\o_{n}^{\a}(e_{n-1})=0,
\;\; \o_{\alpha}^{n-1}(e_{\b})=\o_{\a}^{n}(e_{\b})=0,\end{align}
for $\a,\b=1,\ldots,n-2.$

Let us define distributions  $\mathcal D_{1}$ and $\mathcal D_{2}$ by
\begin{align}\label{4.7} \mathcal D_{1}={\rm Span}\{e_{1},\ldots, e_{n-2}\},\;\; \mathcal D_{2}={\rm Span}\{e_{n-1},e_{n}\}.\end{align}
It follows from \e{4.4} and \e{4.5} that $\mathcal D_{1}$ and  $\mathcal D_{2}$ are integrable distributions such that leaves of $\mathcal D_{1}$ are totally geodesic and leaves of $\mathcal D_{2}$ are totally umbilical in $M^{n}$. Furthermore,  \e{4.2} implies that leaves of $\mathcal D_{2}$ are also totally umbilical in $\mathbb E^{n+1}$. Hence, leaves of $\mathcal D_{2}$ are extrinsic spheres in $\mathbb E^{n+1}$, i.e.,  totally umbilical submanifolds with parallel mean curvature vector field. Now, it is easy to verify that each leaf of $\mathcal D_{2}$ is an extrinsic sphere in $M^{n}$. Consequently, $\mathcal D_{2}$ is a spherical distribution. So, by Hiepko's theorem (cf. \cite[page 90]{book}), $M^{n}$ is locally a warped product $L_{1}\times_{f}L_{2}$, where $L_{1}$ is a leaf of $\mathcal D_{1}$, $L_{2}$ is a leaf of $\mathcal D_{2}$ and $f$ is the warping function.
Since $L_{1}$ is totally geodesic in $M^{n}$ as well as in $\mathbb E^{n+1}$ by \e{4.2}, $L_{1}$ is an open portion of $\mathbb E^{n-2}$. Similarly, since $L_{2}$ is an extrinsic sphere of $\mathbb E^{n+1}$, $L_{2}$ is an open part of a  2-sphere. Thus, without loss of generality, we may assume that the warped product metric of $L_{1}\times_{f}L_{2}$ takes the following form:
\begin{align}\label{4.8} g=\sum_{i=1}^{n-2}dx_{i}^{2}+f^{2}(x_{1},\ldots,x_{n-2})(dx_{n-1}^{2}+\cos^{2}x_{n-1}dx^{2}_{n}).\end{align}
It is easy to see that $\frac{\p}{\p x_{1}},\ldots,\frac{\p}{\p x_{n}}$ are parallel to $e_{1},\ldots,e_{n}$, respectively.

From \e{4.8} we know that the Levi-Civita connection of $g$ satisfies
\begin{equation}\begin{aligned} \label{4.9} &\nabla_{\p_{i}}\p_{j}=0,\;\; i,j=1,\ldots,n-2,
\\&\nabla_{\p_{i}}\p_{n-1}=\frac{f_{i}}{f}\p_{n-1},\;\; \nabla_{\p_{i}}\p_{n}=\frac{f_{i}}{f}\p_{n},
\\& \nabla_{\p_{n-1}}\p_{n-1}=-f\sum_{i=1}^{n-2} f_{i}\p_{i},\;\;
 \nabla_{\p_{n-1}}\p_{n}=-\tan x_{n-1}\p_{n},
\\& \nabla_{\p_{n}}\p_{n}=-f\cos^{2}x_{n-1}\sum_{i=1}^{n-2} f_{i}\p_{i}+\frac{\sin 2x_{n-1}}{2}\p_{n-1},
\end{aligned}\end{equation}
where $\p_{a}=\frac{\p}{\p x_{a}},\, a=1,\ldots,n$ and $f_{i}=\frac{\p f}{\p x_{i}},\, i=1,\ldots,n-2$.

 Gauss' equation and \e{4.2} imply that $\<R(\p_{j},\p_{n-1})\p_{n-1},\p_{j}\>=0$, $j=1,\ldots,n-2$. 
On the other hand, it follows from \e{4.9} that $\<R(\p_{j},\p_{n-1})\p_{n-1},\p_{j}\>=-f f_{jj}.$
Thus, we find $f_{jj}=0$ for $j=1,\ldots,n-2$.

Similarly, we derive from \e{4.2}, \e{4.9}, Gauss' equation and $\<R(\p_{i},\p_{n})\p_{n},\p_{j}\>=0$  $(1\leq i\ne j\leq n-2)$ that $f_{ij}=0$. So, we have $f_{ij}=0$ for $i,j=1,\ldots,n$. Therefore, we obtain
$f=\sum_{i=1}^{n-2}b_{i}x_{i}+c$  for some real numbers $b_{1},\ldots,b_{n-2},c$. Consequently, after applying a suitable rotation and translation, we have $f=a x_{1}$ for some positive number $a$. Thus, \e{4.8} and \e{4.9} become
\begin{align}\label{4.10} g=\sum_{i=1}^{n-2}dx_{i}^{2}+a^{2} x_{1}^{2}(dx_{n-1}^{2}+\cos^{2}x_{n-1}dx^{2}_{n}),\end{align}and
\begin{equation}\begin{aligned} \label{4.11} &\nabla_{\p_{i}}\p_{j}=0,\;\; i,j=1,\ldots,n-2,
\\& \nabla_{\p_{1}}\p_{n-1}=\frac{\p_{n-1}}{x_{1}},\;\; \nabla_{\p_{1}}\p_{n}=\frac{\p_{n}}{x_{1}},
\\& \nabla_{\p_{k}}\p_{n-1}= \nabla_{\p_{k}}\p_{n}=0,\;\; k=2,\ldots,n-2,
\\& \nabla_{\p_{n-1}}\p_{n-1}=-a^{2}x_{1}\p_{1},\;\; \nabla_{\p_{n-1}}\p_{n}=-\tan x_{n-1}\p_{n},
\\& \nabla_{\p_{n}}\p_{n}=-a^{2}x_{1}\cos^{2}x_{n-1}\p_{1}+\frac{\sin 2x_{n-1}}{2}\p_{n-1}.
\end{aligned}\end{equation}

From \e{4.2}, \e{4.10} and Gauss' equation we obtain
\begin{align}\label{4.12} \<R(\p_{n-1},\p_{n})\p_{n},\p_{n-1}\>=a^{4}x_{1}^{4}\lambda^{2}\cos^{2}x_{n-1}.\end{align}
 On the other hand, it follows from \e{4.10} and \e{4.11} that 
\begin{align}\label{4.13} \<R(\p_{n-1},\p_{n})\p_{n},\p_{n-1}\>=a^{2}(1-a^{2})x_{1}^{2}\cos^{2}x_{n-1}.\end{align}
By combining \e{4.12} and \e{4.13} we find $\lambda^{2}=(1-a^{2})/a^{2} x_{1}^{2}$ with $0<a<1$. Thus, without loss of generality, we may put
\begin{align}\label{4.14} \lambda=\frac{\sqrt{1-a^{2}}}{ax_{1}},\end{align}
which shows that $\lambda$ is always nonzero. Hence, $M^{n}$ is non-complete and it contains no minimal points. Moreover, \e{4.14} shows that the immersion of $M^{n}$ in $\mathbb E^{n+1}$ is rigid, since its second fundamental form is completely determined by its metric.

Let $L:M^{n}\to \mathbb E^{n+1}$ denote the immersion of $M^{n}$ in $\mathbb E^{n+1}$. If we put $L_{x_{\a}}=\frac{\p L}{\p x_{\a}}$ and $L_{x_{\a}x_{\b}}=\frac{\p^{2} L}{\p x_{\a}\p x_{\b}}$, then  \e{4.2}, \e{4.10}, \e{4.11}, \e{4.14} and Gauss' formula yield
\begin{equation}\begin{aligned} \label{4.15} &L_{x_{i}x_{j}}=0,\;\; i,j=1,\ldots,n-2,
\\& L_{x_{1}x_{n-1}}=\frac{1}{x_{1}}L_{x_{n-1}},\;\; L_{x_{1}x_{n}}=\frac{1}{x_{1}}L_{x_{n}},
\\& L_{x_{k}x_{n-1}}=L_{x_{k}x_{n}}=0,\;\; k=2,\ldots,n-2,
\\& L_{x_{n-1}x_{n-1}}=-a^{2}x_{1}L_{x_{1}}+a\sqrt{1-a^{2}}\, x_{1} e_{n+1},\;\;
\\& L_{x_{n-1}x_{n}}=-\tan x_{n-1}L_{x_{n}},
\\& L_{x_{n}x_{n}}=-a^{2}x_{1}\cos^{2}x_{n-1}L_{x_{1}}+\frac{\sin 2x_{n-1}}{2}L_{x_{n-1}}\\& \hskip1.2in +a\sqrt{1-a^{2}}\, x_{1} \cos^{2}x_{n-1}e_{n+1}.
\end{aligned}\end{equation} 
Also, from  \e{4.2}, \e{4.14} and Weingarten's formula, we find
\begin{equation}\begin{aligned} \label{4.16} &\frac{\p e_{n+1}}{\p {x_{j}}}=0,\;\;j=1,\ldots,n-2,\\&
 \frac{\p e_{n+1}}{\p {x_{n-1}}}=-\frac{\sqrt{1-a^{2}}}{ax_{1}} L_{x_{n-1}},\;\;\frac{\p e_{n+1}}{\p {x_{n}}}=-\frac{\sqrt{1-a^{2}}}{ax_{1}} L_{x_{n}}.
\end{aligned}\end{equation} 
After solving the PDE system \e{4.15}-\e{4.16}, we obtain
\begin{equation}\begin{aligned} \label{4.17} &L=\sum_{\a=1}^{n-2}c_{\a}x_{\a}   +x_{1}(c_{n-1}\sin x_{n-1}+c_{n}\cos x_{n-1}\sin x_{n}+c_{n+1}\cos x_{n-1}\cos x_{n})
\end{aligned}\end{equation}
for some vectors $c_{1},\ldots,c_{n+1}\in \mathbb E^{n+1}$. Now, by applying \e{4.10} and \e{4.17}, we may conclude that the immersion $L:M^{n}\to \mathbb E^{n+1}$ is congruent to  
\begin{equation}\begin{aligned} &\Big(\! \sqrt{1-a^{2}}\hskip.006in x_{1},x_{2},\ldots,x_{n-2},a x_{1}\sin x_{n-1},\\& \hskip.6in  ax_{1}\cos x_{n-1}\sin x_{n}, a x_{1}\cos x_{n-1}\cos x_{n}\Big).\end{aligned}\end{equation}

Now, let us assume that $M^{n}$ is non-minimal and ideal in $\mathbb E^{m}$ with $m\geq n+2$ and type number $\leq 2$. Then  it follows from Theorem A that  \e{4.2} holds too. Thus, it follows from \e{4.2} and $(\bar\nabla_{e_{\a}}h)(e_{n-1},e_{n-1})=(\bar\nabla_{e_{n-1}}h)(e_{\a},e_{n-1})$ that $D_{e_{\a}}e_{n+1}=0$ for $\a=1,\ldots,n-2.$
Moreover, it follows from \begin{equation}\begin{aligned} \notag &(\bar\nabla_{e_{n}}h)(e_{n-1},e_{n-1})=(\bar\nabla_{e_{n-1}}h)(e_{n-1},e_{n}),\;\; \\& (\bar\nabla_{e_{n-1}}h)(e_{n},e_{n})=(\bar\nabla_{e_{n}}h)(e_{n-1},e_{n})\end{aligned}\end{equation} of Codazzi's equation that $D_{e_{e_{n-1}}}e_{n+1}=D_{e_{n}}e_{n+1}=0$. Thus, we find $De_{n+1}=0$, i.e., $e_{n+1}$ is a parallel normal vector field. Because the first normal bundle is spanned by $e_{n+1}$, it is a parallel normal bundle. Therefore,  Erbarcher's reduction theorem implies that $M^{n}$ is immersed in an $(n+1)$-dimensional affine subspace of $\mathbb E^{m}$. Consequently, we conclude that the immersion is congruent to \e{4.1}.
\end{proof}

 \section{Ideal submanifolds with type number $\leq 2$ in $S^{m}(1)$.}
 
Now, we classify ideal submanifolds of $S^{m}(1)$ with type number $\leq 2$.

\begin{theorem}\label{T:2} Let $M^{n}$ be an ideal submanifold of a unit $m$-sphere $S^{m}(1)$. If the type number of $M^{n}$ in $S^{m}(1)$ is  $\leq 2$, then either $M^{n}$ is a minimal submanifold of $S^{m}(1)$ or the immersion of $M^{n}$ into $S^{m}(1)\subset \mathbb E^{m+1}$ is congruent to
\begin{equation}\begin{aligned} \label{5.1} &\text{\small$\Bigg($}\! \sqrt{1-a^{2}}\hskip.006in \sin x_{1},\cos x_{1 }\sin x_{2},\ldots, \sin x_{n-2}\prod_{j=1}^{n-3}\cos x_{j}, \prod_{j=1}^{n-2}\cos x_{j},\\ & \hskip.1in 
a\sin x_{1}\sin x_{n-1},  a\sin x_{1}\cos x_{n-1}\sin x_{n}, a \sin x_{1}\cos x_{n-1}\cos x_{n},0,\ldots,0\text{\small$\Bigg)$}\end{aligned}\end{equation}
for some real number $a$ satisfying $0< a< 1$.
\end{theorem}
\begin{proof} First,  assume that $M^n$ is a non-minimal ideal hypersurface  of  $S^{n+1}(1)$ with type number $\leq 2$. then as before there is a local orthonormal frame $\{e_{1},\ldots,e_{n}\}$ such that the second fundamental form of $M^n$ in $S^{m}(1)$ satisfies 
\begin{equation}\begin{aligned}\label{5.2} &h(e_{n-1},e_{n-1})=h(e_{n},e_{n})=\lambda e_{n+1},\;\;  h(e_{i},e_{j})=0,\;\; otherwise.\end{aligned}\end{equation}
Thus, by applying the same argument as before, we know that  $M^{n}$ is locally a warped product $L_{1}\times_{f}L_{2}$, where $L_{1}$ is an open portion of a unit $(n-2)$-sphere and $L_{2}$ is an open portion of a 2-sphere. Hence,  the warped product metric of $L_{1}\times_{f}L_{2}$ takes the following form:
\begin{equation}\begin{aligned}\label{5.3} &g=dx_1^2+\cos^2 x_1 dx_2^2+\cdots+\cos^{2} x_{1}\cdots \cos^{2} x_{n-3}dx_{n-2}^2
\\& \hskip.7in +f^{2}(x_{1},\ldots,x_{n-2})(dx_{n-1}^{2}+\cos^{2}x_{n-1}dx^{2}_{n}).\end{aligned}\end{equation}
Clearly,  $\frac{\p}{\p x_{1}},\ldots,\frac{\p}{\p x_{n}}$ are parallel to $e_{1},\ldots,e_{n}$, respectively.
From \e{5.3}, we get
\begin{equation}\begin{aligned} \label{5.4} &\nabla_{\p_{1}}\p_{1}=0,\;\; 
\;\;\nabla_{\p_{i}}\p_{j}=-\tan x_i\p_{j},\;\; 1\leq i<j\leq n-2,
\\&\nabla_{\p_{2}}\p_{2}=\frac{\sin 2x_1}{2}\p_{1},
\\&\nabla_{\p_{k}}\p_{k}=\frac{\sin 2x_1}{2}\Bigg(\prod_{j=2}^{k-1}\cos^2 x_j\! \Bigg)\p_{1}+\frac{\sin 2x_2}{2}\Bigg(\prod_{j=3}^{k-1}\cos^2 x_j\! \Bigg)\p_{2}\\& \hskip1.2in+\cdots +\frac{\sin 2x_{k-1} }{2}\p_{k-1}, \;\;\; k=3,\ldots,n-2,
\\&\nabla_{\p_{i}}\p_{n-1}=\frac{f_{i}}{f}\p_{n-1},\;\; \nabla_{\p_{i}}\p_{n}=\frac{f_{i}}{f}\p_{n},\;\; i=1,\ldots,n-2,
\\
& \nabla_{\p_{n-1}}\p_{n-1}=-f\Bigg\{f_1 \p_1+f_{2}\sec^2 x_1\p_2+\cdots+f_{n-2}\Bigg(\!\prod_{i=1}^{n-3}\sec^2 x_i\! \Bigg)\p_{n-2}\Bigg\} ,\;\;
\\ & \nabla_{\p_{n-1}}\p_{n}=-\tan x_{n-1}\p_{n},
\\& \nabla_{\p_{n}}\p_{n}=-f\cos^{2}x_{n-1}\Bigg\{f_1 \p_1+\cdots
 +f_{n-2}\Bigg(\prod_{i=1}^{n-3}\sec^2 x_i\!\Bigg) \p_{n-2}\Bigg\}.\end{aligned}\end{equation}

From 
$\<R(\p_{j},\p_{n-1})\p_{n-1},\p_{j}\>=\<R(\p_{i},\p_{n})\p_{n},\p_{j}\>=0,\,1\leq i\ne j\leq n-2,$ \e{5.2} and \e{5.4}, we obtain\begin{equation}\begin{aligned} \label{5.5} & f_{11}=-f,\;\; f_{ij}=-\tan x_i f_j,\;\; 1<i<j\leq n-2,
\\&f_{22}=\frac{\sin 2 x_1}{2}f_1-f\cos^2 x_1,
\\&f_{kk}=\frac{\sin 2x_{1}}{2}\cos^{2}x_{2}\cdots \cos^{2}x_{k-1}f_{1}+\cdots +\frac{\sin 2 x_{k-2}}{2}\cos^{2}x_{k-1}f_{k-2}
 \\& \hskip.4in +\frac{\sin 2 x_{k-1}}{2}f_{k-1} -f\prod_{j=1}^{k-1}\cos^2 x_j,\;\;\; k=3,\ldots,n-2.\end{aligned}\end{equation}
After solving system \e{5.5}, we get
\begin{equation}\begin{aligned} \label{5.6} & f(x_1,\ldots,x_{n-2})=b_1\sin x_1
+b_2 \sin x_2\cos x_1+\cdots \\&\hskip.9in +b_{n-2} \sin x_{n-2}\prod_{j=1}^{n-3}\cos x_j +b_{n-1}\prod_{j=1}^{n-2}\cos x_{n-2}
\end{aligned}\end{equation} for some real numbers $b_1,\ldots,b_{n-1}$. The function $f$  in \e{5.6} is the height function of  $S^{n-2}(1)\subset \mathbb E^{n-1}$ in the direction  $v=(b_1,\ldots,b_{n-1})$. So, after applying a suitable rotation of $S^{n-2}(1)$, we get  $f=a \sin x_1$ for some positive number $a$. Thus, taking into account of \e{5.3}  and \e{5.5}, we have
 \begin{equation}\begin{aligned}\label{5.7} &g=dx_1^2+\cos^2 x_1 dx_2^2+\cdots+\cos^{2} x_{1}\cdots\cos^2 x_{n-3} dx_{n-2}^2\\& \hskip1.0in +a^{2}\sin^2 x_1(dx_{n-1}^{2}+\cos^{2}x_{n-1}dx^{2}_{n})\end{aligned}\end{equation}
and
\begin{equation}\begin{aligned} \label{5.8} &\nabla_{\p_{1}}\p_{1}=0,\;\; 
\nabla_{\p_{i}}\p_{j}=-\tan x_i\p_{j},\;\; 1\leq i<j\leq n-2,
\\&\nabla_{\p_{2}}\p_{2}=\sin x_1\cos x_{1}\p_{1},
\\&\nabla_{\p_{k}}\p_{k}=\frac{\sin 2x_1}{2}\cos^{2} x_{2}\cdots \cos^{2}x_{k-1}\p_{1}+\frac{\sin 2x_2}{2}\cos^{3} x_{3}\cdots\cos^2 x_{k-1}\p_{2}\\& \hskip1.2in+\cdots +\sin x_{k-1} \cos x_{k-1}\p_{k-1}, \;\;\; k=3,\ldots,n-2,
\\&\nabla_{\p_{1}}\p_{n-1}=\cot x_1\p_{n-1},\;\; \nabla_{\p_{1}}\p_{n}=\cot x_1\p_{n},
\\&\nabla_{\p_{j}}\p_{n-1}= \nabla_{\p_{j}}\p_{n}=0,\;\; j=2,\ldots,n-2,
\\& \nabla_{\p_{n-1}}\p_{n-1}=-a^2\sin x_1\cos x_{1} \p_1 ,\;\;
 \nabla_{\p_{n-1}}\p_{n}=-\tan x_{n-1}\p_{n},
\\& \nabla_{\p_{n}}\p_{n}=-a^2\sin x_1\cos x_{1} \cos^2 x_{n-1}\p_1+\sin  x_{n-1}\cos x_{n-1}\p_{n-1}.\end{aligned}\end{equation}
By applying Gauss' equation via \e{5.8}, we obtain 
\begin{align}\label{5.9} \lambda=\frac{\sqrt{1-a^{2}}}{a}\csc x_1,\end{align}

Let $L:M^{n}\to S^{n+1}(1)\subset \mathbb E^{n+2}$ be the immersion of $M^{n}$ into $\mathbb E^{n+2}$. We obtain from \e{5.2}, \e{5.7}, \e{5.8}, \e{5.9} and Gauss' formula that 
\begin{equation}\begin{aligned} \label{5.10} &L_{x_{i}x_{j}}=-L,\;\; 
 L_{x_{1}x_{j}}=-\tan x_1 L_j,\;\; j=2,\ldots,n-2,
\\& L_{x_{1}x_{n-1}}=\cot x_1L_{x_{n-1}},\;\; L_{x_{1}x_{n}}=\cot x_1L_{x_{n}},
\\&L_{x_{2}x_{2}}=\sin x_1\cos x_{1}  L_{x_1}-\cos^2 x_1 L,
\\&L_{x_{k}x_{k}}=\frac{\sin 2x_1}{2}\cos^{2} x_{2}\cdots\cos^2 x_{k-1}L_{x_1}\!+\!\frac{\sin 2x_2}{2}\cos^{3} x_{3}\cdots\cos^2 x_{k-1}L_{x_2}\\& \hskip.3in+\cdots +\frac{\sin 2x_{k-1} }{2}L_{x_{k-1}}\! - \cos^{2}x_{1}\cdots \cos^{2}x_{k-1}   L, \;\;\, k=3,\ldots,n-2,
\\& L_{x_{k}x_{n-1}}=L_{x_{k}x_{n}}=0,\;\; k=2,\ldots,n-2,
\\& L_{x_{n-1}x_{n-1}}= -a^2\sin x_1\cos x_{1}  L_{x_{1}}+a\sqrt{1-a^{2}}\,\sin x_{1} e_{n+1}- a^{2}\sin^2 x_1 L,\;\;
\\& L_{x_{n-1}x_{n}}=-\tan x_{n-1}L_{x_{n}},
\\& L_{x_{n}x_{n}}= - a^2\sin x_1\cos x_{1} \cos^{2}x_{n-1}L_{x_{1}}+a^2\sin x_{n-1}\cos x_{n-1} L_{x_{n-1}}\\& \hskip.6in + a\sqrt{1\! -\! a^{2}}\, \sin x_{1} \cos^{2}x_{n-1}e_{n+1} - a^{2}\sin^2 x_1 \cos^{2}x_{n-1} L.
\end{aligned}\end{equation}
Moreover, from \e{5.2}, \e{5.7}, \e{5.9} and Weingarten's formula we have
\begin{equation}\begin{aligned} \label{5.11} &\frac{\p e_{n+1}}{\p {x_{j}}}=0,\;\; j=1,\ldots,n-2,
\\& \frac{\p e_{n+1}}{\p {x_{n-1}}}=-\frac{\sqrt{1-a^{2}}}{a}\csc x_1 L_{x_{n-1}},\;\;\frac{\p e_{n+1}}{\p {x_{n}}}=-\frac{\sqrt{1-a^{2}}}{a}\csc x_1 L_{x_{n}}.\end{aligned}\end{equation}
After solving system \e{5.10}-\e{5.11} we obtain
\begin{equation}\begin{aligned} \label{5.12} &L(x_{1},\ldots,x_{n})=c_{1}\sin x_{1}
+c_{2}\sin x_{2}\cos x_{1}+\cdots \\&\hskip1.1in +c_{n-2}\sin x_{n-2}\prod_{j=1}^{n-3}\cos x_{j}  +c_{n-1}\prod_{j=1}^{n-2}\cos x_{j}
\\&\hskip.3in + \sin x_{1}(c_{n}\sin x_{n-1}+c_{n+1}\cos x_{n-1}\sin x_{n}+c_{n+2}\cos x_{n-1}\cos x_{n})
\end{aligned}\end{equation}
for some vectors $c_{1},\ldots,c_{n+1}\in \mathbb E^{n+2}$. Therefore, after applying \e{5.7} and \e{5.12}, we conclude that $L$ is congruent to  
\begin{equation}\begin{aligned}  \notag&\text{\small$\Bigg($}\! \sqrt{1-a^{2}}\hskip.006in \sin x_{1},\cos x_{1 }\sin x_{2},\ldots, \sin x_{n-2}\prod_{j=1}^{n-3}\cos x_{j}, \prod_{j=1}^{n-2}\cos x_{j},\\ & \hskip.3in 
a\sin x_{1}\sin x_{n-1},  a\sin x_{1}\cos x_{n-1}\sin x_{n}, a \sin x_{1}\cos x_{n-1}\cos x_{n}\text{\small$\Bigg)$}.\end{aligned}\end{equation}

If $M^{n}$ is non-minimal and ideal in $S^{m}(1)$ with $m\geq n+2$ and type number $\leq 2$.  Then by applying  Theorem A and Codazzi's equation we know that  the first normal bundle is a parallel normal bundle. Therefore, the reduction theorem implies that $M^{n}$ is immersed in a totally geodesic  $S^{n+1}(1)\subset S^{m}(1)$. Consequently, the immersion is congruent to \e{5.1}.
\end{proof}

 \section{Ideal submanifolds with type number $\leq 2$ in $H^{m}(-1)$.}
 
 Let $\mathbb E_{1}^{m+1}$ denote the $(m+1)$-dimensional Minkowski spacetime endowed with the Lorentzian metric 
\begin{align}g= - du_{i}^{2} + \sum_{j=2}^{m+1}du_{j}^{2}. \end{align}
 We put
$$H^m(-1)=\left\{u=(u_{1},\ldots,u_{m+1})\in \mathbb E^{m+1}_{1}: g(u,u)=-1 \; {\rm and}\; u_{1}>0\right\}.$$
Then  $H^m(-1)$ is the hyperbolic $m$-space of constant sectional curvature $-1$.
 
In this section, we classify  ideal submanifolds of $H^{m}(-1)$ with type number $\leq 2$.

\begin{theorem} \label{T:3}  Let $\phi:M^{n}\to H^{m}(-1)$ be an ideal immersion of a Riemannian $n$-manifold into $H^{m}(-1)$. If $M^{n}$ has type number $\leq 2$, then either $M^{n}$ is  minimal in $H^{m}(-1)$ or $\phi(M^{n})$ lies in a totally geodesic $H^{n+1}(-1)\subset H^{m}(-1)$. Moreover, in the later case the corresponding immersion $L:M^{n}\to H^{n+1}(-1)\subset \mathbb E^{n+2}_{1}$ is congruent to one  of the following three immersions:

$$\leqno{(A)}\hskip.1in\left\{\begin{aligned}  \Bigg( &\frac{ab \sinh x_{n-2}+(1+b^{2})\cosh x_{n-2}}{\sqrt{1+b^{2}}}\prod_{j=1}^{n-3}\cosh x_{j},\sinh x_{1},\sinh x_{2}\cosh x_{1}, 
\\  &\hskip.2in \ldots,\, \sinh x_{n-3}\prod_{j=1}^{n-4}\cosh x_{j},
 \frac{\sqrt{1-a^{2}+b^{2}}}{\sqrt{1+b^{2}}}\sinh x_{n-2}\prod_{j=1}^{n-3}\cosh x_{j},
 \\ & \hskip.3in  (a \sinh x_{n-2}\!+\! b\cosh x_{n-2})\cos x_{n-1}\cos x_{n}\!\prod_{j=1}^{n-3}\!\cosh x_j,
  \\& \hskip.4in (a \sinh x_{n-2}+b\cosh x_{n-2})\cos x_{n-1}\sin x_{n}\prod_{j=1}^{n-3}\cosh x_j , 
  \\ &\hskip.2in   (a \sinh x_{n-2}+b\cosh x_{n-2}) \sin x_{n-1}\prod_{j=1}^{n-3}\cosh x_j \Bigg),\; \; a^{2}<1+b^{2};\end{aligned}\right.$$
$$\leqno{(B)}\hskip.1in\left\{\begin{aligned}  \Bigg( &\frac{a(b^{4}\!-\!4 \!+\!4b^{2}(x_{n-1}^{2}\!+\!x_{n}^{2}))\sinh x_{n-2}+b(b^{4}\!+\!4 \!+\!4b^{2}(x_{n-1}^{2}\!+\!x_{n}^{2}))\cosh x_{n-2}}{4b^{3}\prod_{j=1}^{n-3}\sech\, x_{j}},
\\& \frac{a(b^{4}\!+\!4 \!-\!4b^{2}(x_{n-1}^{2}\!+\!x_{n}^{2}))\sinh x_{n-2}+b(b^{4}\!-\!4 \!-\!4b^{2}(x_{n-1}^{2}\!+\!x_{n}^{2}))\cosh x_{n-2}}{4b^{3}\prod_{j=1}^{n-3}\sech\, x_{j}},
\\& \hskip.3in \sinh x_{1}, \, \ldots,\, \sinh x_{n-3}\prod_{j=1}^{n-4}\cosh x_{j},\frac{\sqrt{b^{2}\!-\!a^{2}}}{b}\sinh x_{n-2}\prod_{j=1}^{n-3}\cosh x_{j},
 \\& \hskip.5in  (a \sinh x_{n-2}\!+\! b\cosh x_{n-2}) x_{n-1}\!\prod_{j=1}^{n-3}\!\cosh x_j,
  \\& \hskip.6in (a \sinh x_{n-2}+b\cosh x_{n-2}) x_{n}\prod_{j=1}^{n-3}\cosh x_j  \Bigg),\;\;  a^{2}<b^{2};\end{aligned}\right.$$
$$\leqno{(C)}\hskip.1in\left\{\begin{aligned}   & \Bigg(\! (a\sinh x_{n-2}+b\cosh x_{n-2})\cosh x_{n-1}\cosh x_{n}\! \prod_{j=1}^{n-3}\! \cosh x_{j},
\\& \hskip.4in (a\sinh x_{n-2}+b\cosh x_{n-2})\cosh x_{n-1}\sinh x_{n}\! \prod_{j=1}^{n-3}\! \cosh x_{j},
\\&\hskip.6in  (a\sinh x_{n-2}+b\cosh x_{n-2})\sinh x_{n-1}\! \prod_{j=1}^{n-3}\! \cosh x_{j},
\\& \hskip.4in \sinh x_{1}, \sinh x_{2} \cosh x_{1},\,\ldots,\, \sinh x_{n-3}\! \prod_{{j=1}}^{n-4}\! \cosh x_{j},
 \\& \hskip.1in \frac{\sqrt{b^{2}-a^{2}-1}}{\sqrt{1+a^{2}}}\prod_{j=1}^{n-2} \!\cosh x_{j},
 \frac{ab\cosh x_{n-2}+ (1+a^{2})\sinh x_{n-2}}{\sqrt{1+a^{2}}}  \prod_{j=1}^{n-3}\! \cosh x_{j}\! \Bigg),\; \;\\&\hskip3in  1+a^{2}<b^{2}.\end{aligned}\right.$$

\end{theorem}
\begin{proof} Assume that $M^n$ is a non-minimal ideal hypersurface  of  $H^{n+1}(-1)$ with type number $\leq 2$. Then there is a local orthonormal frame $\{e_{1},\ldots,e_{n}\}$ such that 
\begin{equation}\begin{aligned}\label{6.2} &h(e_{n-1},e_{n-1})=h(e_{n},e_{n})=\lambda e_{n+1},\;\; h(e_{i},e_{j})=0,\;\; otherwise.\end{aligned}\end{equation}
By applying the same argument as before, we know that  $M^{n}$ is locally a warped product $L_{1}\times_{f}L_{2}$, where $L_{1}$ is an open portion of $H^{n-2}(-1)$ and $L_{2}$ is a totally umbilical surface of $H^{n+1}(-1)$. Therefore, the warped product metric of $L_{1}\times_{f}L_{2}$ takes one of the following three forms:
\begin{align}\label{6.3} &g_{1}=dx_1^2+\cosh^2 x_1 dx_2^2+\cdots+\cosh^{2} x_{1}\cdots \cosh^{2} x_{n-3}dx_{n-2}^2
\\ \notag &\hskip.5in +f^{2}(x_{1},\ldots,x_{n-2})(dx_{n-1}^{2}+\cos^{2}x_{n-1}dx^{2}_{n});
\\ \label{6.4} &g_{2}=dx_1^2+\cosh^2 x_1 dx_2^2+\cdots+\cosh^{2} x_{1}\cdots \cosh^{2} x_{n-3}dx_{n-2}^2
\\ \notag &\hskip.5in +f^{2}(x_{1},\ldots,x_{n-2})(dx_{n-1}^{2}+dx^{2}_{n});\end{align}\begin{align}
\label{6.5} &g_{3}=dx_1^2+\cosh^2 x_1 dx_2^2+\cdots+\cosh^{2} x_{1}\cdots \cosh^{2} x_{n-3}dx_{n-2}^2
\\ \notag &\hskip.5in +f^{2}(x_{1},\ldots,x_{n-2})(dx_{n-1}^{2}+\cosh^{2}x_{n-1}dx^{2}_{n}).\end{align}

{\it Case} (1):  {\it The metric of $L_{1}\times_{f}L_{2}$ is}  \e{6.3}. In this case, we find \begin{equation}\begin{aligned} \label{6.6} &\nabla_{\p_{1}}\p_{1}=0,\;\; 
\;\; \\& \nabla_{\p_{i}}\p_{j}=\tanh x_i\p_{j},\;\; 1\leq i<j\leq n-2,
\\&\nabla_{\p_{2}}\p_{2}=-\sinh x_1\cosh x_{1}\p_{1},
\\&\nabla_{\p_{k}}\p_{k}=-\frac{\sinh 2x_1}{2}\prod_{j=2}^{k-1}\cosh^2 x_j \p_{1}-\frac{\sinh 2x_2}{2}\prod_{j=3}^{k-1}\cosh^2 x_j \p_{2}
\\& \hskip.9in -\cdots -\sinh x_{k-1}\cosh x_{k-1}\p_{k-1}, \;\;\; k=3,\ldots,n-2,
\\&\nabla_{\p_{i}}\p_{n-1}=\frac{f_{i}}{f}\p_{n-1},\;\; \nabla_{\p_{i}}\p_{n}=\frac{f_{i}}{f}\p_{n},\;\; i=1,\ldots,n-2,
\\& \nabla_{\p_{n-1}}\p_{n-1}=-f\Big\{f_1 \p_1\!+\!f_{2}\,\sech^2 x_1\p_2+\cdots+\!  f_{n-2}\prod_{i=1}^{n-3}\sech^2 x_i\p_{n-2}\Big\} ,\;\;
\\& \nabla_{\p_{n-1}}\p_{n}=-\tan x_{n-1}\p_{n},
\\& \nabla_{\p_{n}}\p_{n}=-f\cos^{2}x_{n-1}\Big\{f_1 \p_1+\cdots+\!  f_{n-2}\prod_{i=1}^{n-3}\sech^2 x_i\p_{n-2}\Big\}.\end{aligned}\end{equation}

From 
$\<R(\p_{j},\p_{n-1})\p_{n-1},\p_{j}\>=\<R(\p_{i},\p_{n})\p_{n},\p_{j}\>=0,\,1\leq i\ne j\leq n-2,$ \e{6.2} and \e{6.6}, it follows that 
\begin{equation}\begin{aligned} \label{6.7} & f_{11}=f,\;\; f_{ij}=\tanh x_i f_j,\;\; 1<i<j\leq n-2,
\\&f_{22}+\sinh x_1\cosh x_{1} f_1=f\cosh^2 x_1,
\\&f_{kk}+f_{1}\sinh x_1\cosh x_{1}\cosh^{2}x_{2}\cdots \cosh^{2}x_{k-1}+\cdots  \\& \hskip.1in+f_{k-2}\sinh x_{k-2}\cosh x_{k-2}\cosh^{2}x_{k-1}
 +f_{k-1} \sinh x_{k-1}\cosh x_{k-1} \\&\hskip.2in =f\cosh^2 x_1\cdots \cosh^{2}x_{k-1},\;\;\; k=3,\ldots,n-2.\end{aligned}\end{equation} 
After solving system \e{6.7}, we obtain
\begin{equation}\begin{aligned} \label{6.8} & f(x_1,\ldots,x_{n-2})=b_1\sinh x_1
+b_2 \sinh x_2\cosh x_1+\cdots \\&\hskip.4in +b_{n-2} \sinh x_{n-2}\prod_{j=1}^{n-3}\cosh x_j +b_{n-1}\prod_{j=1}^{n-2}\cosh x_{j}
\end{aligned}\end{equation} for some real numbers $b_1,\ldots,b_{n-1}$.  So, after applying a suitable rotation of  $H^{n+1}(-1)$,  we get 
\begin{align}f=(a \sinh x_{n-2}+b\cosh x_{n-2})\prod_{j=1}^{n-3}\cosh x_j\end{align} for some real numbers $a$ and $b$, not both zero.
Consequently, we obtain 
\begin{equation}\begin{aligned}\label{6.10} &g_{1}=dx_1^2+\cosh^2 x_1 dx_2^2+\cdots+\cosh^{2} x_{1}\cdots \cosh^{2} x_{n-3}dx_{n-2}^2
\\  &\hskip.1in +(a \sinh x_{n-2}+b\cosh x_{n-2})^{2}\prod_{j=1}^{n-3}\cosh^{2} x_j (dx_{n-1}^{2}+\cos^{2}x_{n-1}dx^{2}_{n})\end{aligned}\end{equation}
which implies that the Levi-Civita connection satisfies
\begin{equation}\begin{aligned} \label{6.11} &\nabla_{\p_{1}}\p_{1}=0,\;\; 
\nabla_{\p_{i}}\p_{j}=\tanh x_i\p_{j},\;\; 1\leq i<j\leq n-2,
\\& \nabla_{\p_{2}}\p_{2}=-\frac{\sinh 2x_1}{2}\p_{1},
\\&\nabla_{\p_{k}}\p_{k}=-\sum_{j=1}^{k-1}\frac{\sinh 2x_j}{2}\prod_{i=j+1}^{k-1}\! \cosh^{2} x_{i}\p_{j}, \;\; k=3,\ldots,n-2, 
\\&\nabla_{\p_{s}}\p_{r}=\tanh x_{s} \p_{r},\;\; s=1,\ldots,n-3, \;\; r=n-1,n,\\&\nabla_{\p_{n-2}}\p_{r}=\frac{a \cosh x_{n-2}+b\sinh x_{n-2}}{a \sinh x_{n-2}+b\cosh x_{n-2}} \p_{r},\;\; r=n-1,n, 
\\& \nabla_{\p_{n-1}}\p_{n-1}=-(a \sinh x_{n-2}\!+\!b\cosh x_{n-2})^{2}\Bigg\{\frac{\sinh 2x_{1}}{2}\prod_{j=2}^{n-3}\cosh^{2}x_{j} \p_{1} \\& \hskip.6in+ \frac{\sinh 2x_{2}}{2}\prod_{j=3}^{n-3}\cosh^{2}x_{j} \p_{2}+\cdots +
\frac{\sinh 2x_{n-3}}{2} \p_{n-3}\Bigg\}
\\& \hskip.6in -(a \sinh x_{n-2}\!+\!b\cosh x_{n-2})(a \cosh x_{n-2}\!+\!b\sinh x_{n-2}) \p_{n-2},\;\;
\\& \nabla_{\p_{n-1}}\p_{n}=-\tan x_{n-1}\p_{n},
\\& \nabla_{\p_{n}}\p_{n}=-(a \sinh x_{n-2}\!+\!b\cosh x_{n-2})^{2}\cos^{2} x_{n-1}\Bigg\{\! \frac{\sinh 2x_{1}}{2}\! \prod_{j=2}^{n-3} \cosh^{2}x_{j} \p_{1} \\& \hskip.6in+ \frac{\sinh 2x_{2}}{2}\prod_{j=3}^{n-3}\cosh^{2}x_{j} \p_{2}+\cdots +
\frac{\sinh 2x_{n-3}}{2} \p_{n-3}\Bigg\}
\\& \hskip.2in -(a \sinh x_{n-2}\!+\!b\cosh x_{n-2})(a \cosh x_{n-2}\!+\!b\sinh x_{n-2})\cos^{2} x_{n-1} \p_{n-2}\\& \hskip1.4in +\frac{\sin 2 x_{n-1}}{2}\p_{n-1}.
\end{aligned}\end{equation}

From the equation $\<R(\p_{n-1},\p_{n})\p_{n},\p_{n-1}\>=(\lambda^{2}-1)g_{n-1n-1}g_{nn}$
 of Gauss and  \e{6.11} we find
\begin{align}\label{6.12}\lambda^{2}=\frac{(1-a^{2}+b^{2})}{(a \sinh x_{n-2}+b \cosh x_{n-2})^{2}}\prod_{j=1}^{n-3} \,\sech^{2}x_{j},\end{align} which implies $a^{2}<1+b^{2}$.
Thus, we may put
\begin{align}\label{6.13} \lambda=\frac{\sqrt{1-a^{2}+b^{2}}}{a \sinh x_{n-2}+b \cosh x_{n-2}}\prod_{j=1}^{n-3} \,\sech x_{j},\;\;\; a^{2}<1+b^{2}.\end{align}
Hence, $M^{n}$ does not contain minimal points in $H^{n+1}(-1)$.

Let $L:M^{n}\to H^{n+1}(-1)\subset \mathbb E_{1}^{n+2}$ denote the immersion of $M^{n}$ into $\mathbb E_{1}^{n+2}$. Then we get from \e{6.2}, \e{6.10}, \e{6.11}, \e{6.13} and Gauss' formula that
\begin{equation}\begin{aligned}\label{6.14} &L_{x_{i}x_{j}}=L,\;\; 
L_{x_{1}x_{j}}=\tanh x_1 L_j,\;\; j=2,\ldots,n-2,
\\&  L_{x_{2}x_{2}}=-\sinh x_1\cosh x_{1} L_{x_1}+\cosh^2 x_1 L,
\\&  L_{x_{k}x_{k}}=\prod_{j=1}^{k-1}\cosh^{2} x_{j}L-\sum_{j=1}^{k-1}\frac{\sinh 2x_j}{2}\! \prod_{i=j+1}^{k-1}\! \cosh^{2} x_{i}L_{j}, \; k=3,\ldots,n\!-\!2, 
\\&L_{x_{s}x_{r}}=\tanh x_{s} L_{x_{r}},\;\; s=1,\ldots,n-3, \;\; r=n-1,n,
\\&L_{x_{n-2}x_{r}} =\frac{a \cosh x_{n-2}+b\sinh x_{n-2}}{a \sinh x_{n-2}+b\cosh x_{n-2}} L_{x_{r}},
\;\; r=n-1,n, 
\\&  L_{x_{n-1}x_{n-1}}= (a \sinh x_{n-2}+b\cosh x_{n-2})^{2}\prod_{j=1}^{n-3}\cosh^{2} x_j  L\\&\hskip.4in -(a \sinh x_{n-2}\!+\!b\cosh x_{n-2})^{2}\Bigg\{\frac{\sinh 2x_{1}}{2} \prod_{j=2}^{n-3}\cosh^{2}x_{j}  L_{x_{1}} \\& \hskip.6in+ \frac{\sinh 2x_{2}}{2} \prod_{j=3}^{n-3}\cosh^{2}x_{j} L_{x_{2}}+\cdots +
\frac{\sinh 2x_{n-3}}{2} L_{x_{n-3}}\Bigg\}
\\& \hskip.6in -(a \sinh x_{n-2}\!+\!b\cosh x_{n-2})(a \cosh x_{n-2}\!+\!b\sinh x_{n-2})L_{x_{n-2}}
\\&\hskip.4in +\sqrt{1-a^{2}+b^{2}} (a \sinh x_{n-2}+b\cosh x_{n-2})\prod_{j=1}^{n-3}\cosh x_j e_{n+1},
\\&  L_{x_{n-1}x_{n}}=-\tan x_{n-1}L_{x_{n}},
\\&  L_{x_{n}x_{n}}=(a \sinh x_{n-2}+b\cosh x_{n-2})^{2}\cos^{2} x_{n-1}\prod_{j=1}^{n-3}\cosh^{2} x_j  L\\&\hskip.1in -(a \sinh x_{n-2}\!+\!b\cosh x_{n-2})^{2}\cos^{2} x_{n-1}\Bigg\{\! \frac{\sinh 2x_{1}}{2}\! \prod_{j=2}^{n-3} \cosh^{2}x_{j} L_{x_{1}} \\& \hskip.6in+ \frac{\sinh 2x_{2}}{2}\prod_{j=3}^{n-3}\cosh^{2}x_{j} L_{x_{2}}+\cdots +
\frac{\sinh 2x_{n-3}}{2} L_{x_{n-3}}\Bigg\}
\\& \hskip.2in -(a \sinh x_{n-2}\!+\!b\cosh x_{n-2})(a \cosh x_{n-2}\!+\!b\sinh x_{n-2})\cos^{2} x_{n-1} L_{x_{n-2}}
\end{aligned}\end{equation}\begin{equation}\begin{aligned}\notag
\\& \hskip1.4in +\sin  x_{n-1}\cos x_{n-1}L_{x_{n-1}}
\\&\hskip.2in +\sqrt{1-a^{2}+b^{2}} \,(a \sinh x_{n-2}+b\cosh x_{n-2})\cos^{2}x_{n-1}\prod_{j=1}^{n-3}\cosh x_j e_{n+1}
.\end{aligned}\end{equation}
Moreover, it follows from \e{6.2}, \e{6.10}, \e{6.13} and Weingarten's formula that
\begin{equation}\begin{aligned} \label{6.15} &\frac{\p e_{n+1}}{\p {x_{j}}}=0,\;\; j=1,\ldots,n-2,
\\& \frac{\p e_{n+1}}{\p {x_{r}}}=-\frac{\sqrt{1-a^{2}+b^{2}}\,\sech\, x_{1}\cdots\, \sech\, x_{{n-3}}}{a \sinh x_{n-2}+b \cosh x_{n-2}} L_{x_{r}},\;\; r=n-1,n.\end{aligned}\end{equation}
Solving system \e{6.14}-\e{6.15} gives 
\begin{equation}\begin{aligned} \label{6.16} &  L(x_{1},\ldots,x_{n})=c_{1}\sinh x_{1}+\cosh x_{1}\Big\{
c_{2}\sinh x_{2}+c_{3}\sinh x_{3}\cosh x_{2} 
\\&\hskip.4in +\cdots+c_{n-2}\sinh x_{n-2}\prod_{j=2}^{n-3}\cosh x_{j}  +c_{n-1}\prod_{j=2}^{n-2}\cosh x_{j}\Big\}
\\&\hskip.5in +(a \sinh x_{n-2}+b\cosh x_{n-2})\prod_{j=1}^{n-3}\cosh x_j \big\{c_{n}\sin x_{n-1}
\\& \hskip.9in +c_{n+1}\cos x_{n-1}\sin x_{n}+c_{n+2}\cos x_{n-1}\cos x_{n}\big\}.
\end{aligned}\end{equation} 
Now, we conclude from \e{6.10} and \e{6.16} that $L$ is congruent to immersion (A).

If $M^{n}$ is non-minimal and ideal in $H^{m}(-1)$ with type number $\leq 2$, then it follows from Theorem A, Codazzi's equation and Reduction Theorem that  $M^{n}$ is immersed in a totally geodesic  $H^{n+1}(-1)\subset H^{m}(-1)$. Therefore, we obtain the same conclusion.
\vskip.1in

{\it Case} (2):  {\it The metric of $L_{1}\times_{f}L_{2}$ is}  \e{6.4}. In this case, we find in the same way as case (1) that the warping function is given by \e{6.8}. Thus, without loss of generality, we may assume that the metric tensor is given by
\begin{equation}\begin{aligned}\label{6.17} &g_{2}=dx_1^2+\cosh^2 x_1 dx_2^2+\cdots+\cosh^{2} x_{1}\cdots \cosh^{2} x_{n-3}dx_{n-2}^2
\\  &\hskip.1in +(a \sinh x_{n-2}+b\cosh x_{n-2})^{2}\prod_{j=1}^{n-3}\cosh^{2} x_j (dx_{n-1}^{2}+dx^{2}_{n}).\end{aligned}\end{equation}
Now, by applying a similar argument as case (1), we get
\begin{align}\label{6.18}\lambda= \frac{\sqrt{b^{2}-a^{2}}\, }{a \sinh x_{n-2}+b \cosh x_{n-2}}\prod_{j=1}^{n-3}\sech\, x_{j}\end{align}
for some real numbers $a,b$ satisfying $b^{2}>a^{2}$. Hence, we  obtain the following system for the immersion $L: M^{n}\to H^{n+1}(-1)\subset \mathbb E^{n+2}_{1}$.

\begin{equation}\begin{aligned}\label{6.19} &L_{x_{i}x_{j}}=L,\;\; 
L_{x_{1}x_{j}}=\tanh x_1 L_j,\;\; j=2,\ldots,n-2,
\\&  L_{x_{2}x_{2}}=-\sinh x_1\cosh x_{1} L_{x_1}+\cosh^2 x_1 L,
\\&  L_{x_{k}x_{k}}=\prod_{j=1}^{k-1}\cosh^{2} x_{j}L-\sum_{j=1}^{k-1}\frac{\sinh 2x_j}{2}\! \prod_{i=j+1}^{k-1}\! \cosh^{2} x_{i}L_{j}, \; k=3,\ldots,n\!-\!2, 
\\&L_{x_{s}x_{r}}=\tanh x_{s} L_{x_{r}},\;\; s=1,\ldots,n-3, \;\; r=n-1,n,
\\&L_{x_{n-2}x_{r}} =\frac{a \cosh x_{n-2}+b\sinh x_{n-2}}{a \sinh x_{n-2}+b\cosh x_{n-2}} L_{x_{r}},
\;\; r=n-1,n, 
\\&  L_{x_{n-1}x_{n-1}}= (a \sinh x_{n-2}+b\cosh x_{n-2})^{2}\prod_{j=1}^{n-3}\cosh^{2} x_j  L\\&\hskip.4in -(a \sinh x_{n-2}\!+\!b\cosh x_{n-2})^{2}\Bigg\{\frac{\sinh 2x_{1}}{2} \prod_{j=2}^{n-3}\cosh^{2}x_{j}  L_{x_{1}}
\\& \hskip.6in+ \frac{\sinh 2x_{2}}{2} \prod_{j=3}^{n-3}\cosh^{2}x_{j} L_{x_{2}}+\cdots +
\frac{\sinh 2x_{n-3}}{2} L_{x_{n-3}}\Bigg\}
\\& \hskip.3in -(a \sinh x_{n-2}\!+\!b\cosh x_{n-2})(a \cosh x_{n-2}\!+\!b\sinh x_{n-2})L_{x_{n-2}}
\\&\hskip.4in +\sqrt{b^{2}-a^{2}} (a \sinh x_{n-2}+b\cosh x_{n-2})\prod_{j=1}^{n-3}\cosh x_j e_{n+1},
\\&  L_{x_{n-1}x_{n}}=0,
\\&  L_{x_{n}x_{n}}=(a \sinh x_{n-2}+b\cosh x_{n-2})^{2}\prod_{j=1}^{n-3}\cosh^{2} x_j  L\\&\hskip.1in -(a \sinh x_{n-2}\!+\!b\cosh x_{n-2})^{2}\Bigg\{ \frac{\sinh 2x_{1}}{2}\! \prod_{j=2}^{n-3} \cosh^{2}x_{j} L_{x_{1}} \\& \hskip.6in+ \frac{\sinh 2x_{2}}{2}\prod_{j=3}^{n-3}\cosh^{2}x_{j} L_{x_{2}}+\cdots +
\frac{\sinh 2x_{n-3}}{2} L_{x_{n-3}}\Bigg\}
\\& \hskip.2in -(a \sinh x_{n-2}\!+\!b\cosh x_{n-2})(a \cosh x_{n-2}\!+\!b\sinh x_{n-2}) L_{x_{n-2}}
\\&\hskip.4in +\sqrt{b^{2}-a^{2}} \,(a \sinh x_{n-2}+b\cosh x_{n-2})\prod_{j=1}^{n-3}\cosh x_j e_{n+1}
,\end{aligned}\end{equation}
and\begin{equation}\begin{aligned} \label{6.20} &\frac{\p e_{n+1}}{\p {x_{j}}}=0,\;\; j=1,\ldots,n-2,
\\& \frac{\p e_{n+1}}{\p {x_{r}}}=-\frac{\sqrt{b^{2}-a^{2}}\,\sech\, x_{1}\cdots\, \sech\, x_{{n-3}}}{a \sinh x_{n-2}+b \cosh x_{n-2}} L_{x_{r}},\;\; r=n-1,n.\end{aligned}\end{equation}
After solving system \e{6.19}-\e{6.20}, we obtain 

\begin{equation}\begin{aligned} &\hskip.0in   L(x_{1},\ldots,x_{n})=c_{1}\sinh x_{1}+\cosh x_{1}\Bigg\{c_{2}\sinh x_{2}+c_{3}\sinh x_{3}\cosh x_{2} 
\\& \label{6.21}\hskip.6in +\cdots+c_{n-2}\sinh x_{n-2}\prod_{j=2}^{n-3}\cosh x_{j}  +c_{n-1}\prod_{j=2}^{n-2}\cosh x_{j}\Bigg\}
\\& \hskip.5in +(a \sinh x_{n-2}\!+\!b \cosh x_{n-2})\Bigg(\prod_{j=1}^{n-3}\cosh x_{j}\Bigg)\times \\&\hskip1.1in \big\{c_{n} x_{n-1}+c_{n+1}x_{n}+c_{n+2}(x_{n-1}^{2}\!+\!x_{n}^{2})\big\}.\end{aligned}\end{equation} 
So, by using \e{6.17} and \e{6.21} we conclude that  $L$ is congruent to  immersion  (B).

\vskip.05in
{\it Case} (3):  {\it The metric of $L_{1}\times_{f}L_{2}$ is}  \e{6.5}. In this case, we find in the same way as case (1) that the warping function is given by \e{6.8}. Thus, without loss of generality, we may assume that the metric tensor is given by

\begin{equation}\begin{aligned}\label{6.22} &g_{3}=dx_1^2+\cosh^2 x_1 dx_2^2+\cdots+\cosh^{2} x_{1}\cdots \cosh^{2} x_{n-3}dx_{n-2}^2
\\  &\hskip.1in +(a \sinh x_{n-2}+b\cosh x_{n-2})^{2}\prod_{j=1}^{n-3}\cosh^{2} x_j (dx_{n-1}^{2}+\cosh^{2}x_{n-1}dx^{2}_{n}).\end{aligned}\end{equation}
Now, by applying a similar argument as case (1), we get
\begin{align}\label{6.23}\lambda= \frac{\sqrt{b^{2}-a^{2}-1}\, }{a \sinh x_{n-2}+b \cosh x_{n-2}}\prod_{j=1}^{n-3}\sech\, x_{j},\;\; b^{2}>1+a^{2}.\end{align}
 Thus, we may derive the following system for $L: M^{n}\to H^{n+1}(-1)\subset \mathbb E^{n+2}_{1}$.
\begin{equation}\begin{aligned}\label{6.24} &L_{x_{i}x_{j}}=L,\;\; 
L_{x_{1}x_{j}}=\tanh x_1 L_j,\;\; j=2,\ldots,n-2,
\\&  L_{x_{2}x_{2}}=-\sinh x_1\cosh x_{1} L_{x_1}+\cosh^2 x_1 L,
\\&  L_{x_{k}x_{k}}=\prod_{j=1}^{k-1}\cosh^{2} x_{j}L-\sum_{j=1}^{k-1}\frac{\sinh 2x_j}{2}\! \prod_{i=j+1}^{k-1}\! \cosh^{2} x_{i}L_{j}, \; k=3,\ldots,n\!-\!2, 
\\&L_{x_{s}x_{r}}=\tanh x_{s} L_{x_{r}},\;\; s=1,\ldots,n-3, \;\; r=n-1,n,
\\&L_{x_{n-2}x_{r}} =\frac{a \cosh x_{n-2}+b\sinh x_{n-2}}{a \sinh x_{n-2}+b\cosh x_{n-2}} L_{x_{r}},
\;\; r=n-1,n, 
\\&  L_{x_{n-1}x_{n-1}}= (a \sinh x_{n-2}\!+\! b\cosh x_{n-2})^{2}\Bigg\{\! \prod_{j=1}^{n-3}\cosh^{2} x_j  L \!-\! \frac{\sinh 2x_{1}}{2}\! \prod_{j=2}^{n-3}\!\cosh^{2}x_{j}  L_{x_{1}}\\&\hskip.4in- \frac{\sinh 2x_{2}}{2} \prod_{j=3}^{n-3}\cosh^{2}x_{j} L_{x_{2}}-\cdots -
\frac{\sinh 2x_{n-3}}{2} L_{x_{n-3}}\Bigg\}
\\& \hskip.6in -(a \sinh x_{n-2}\!+\!b\cosh x_{n-2})(a \cosh x_{n-2}\!+\!b\sinh x_{n-2})L_{x_{n-2}}
\\&\hskip.4in +\sqrt{b^{2}-a^{2}-1} (a \sinh x_{n-2}+b\cosh x_{n-2})\prod_{j=1}^{n-3}\cosh x_j e_{n+1},
\end{aligned}\end{equation}\begin{equation}\begin{aligned}\notag
&  L_{x_{n-1}x_{n}}=\tanh x_{n-1}L_{x_{n}},
\\&  L_{x_{n}x_{n}}=(a \sinh x_{n-2}+b\cosh x_{n-2})^{2}\cosh^{2} x_{n-1}\Bigg\{\prod_{j=1}^{n-3}\cosh^{2} x_j  L\\&\hskip.1in - \frac{\sinh 2x_{1}}{2}\! \prod_{j=2}^{n-3} \cosh^{2}x_{j} L_{x_{1}}\!-\! \frac{\sinh 2x_{2}}{2}\prod_{j=3}^{n-3}\cosh^{2}x_{j} L_{x_{2}}\!-\cdots -\!
\frac{\sinh 2x_{n-3}}{2} L_{x_{n-3}}\!\Bigg\}
\\& \hskip.2in -(a \sinh x_{n-2}\!+\!b\cosh x_{n-2})(a \cosh x_{n-2}\!+\!b\sinh x_{n-2})\cosh^{2} x_{n-1} L_{x_{n-2}}\\& \hskip1.4in -\sinh  x_{n-1}\cosh x_{n-1}L_{x_{n-1}}
\\&\hskip.2in +\sqrt{b^{2}-a^{2}-1} \,(a \sinh x_{n-2}+b\cosh x_{n-2})\cosh^{2}x_{n-1}\prod_{j=1}^{n-3}\cosh x_j e_{n+1},
\end{aligned}\end{equation}
and
\begin{equation}\begin{aligned} \label{6.25} &\frac{\p e_{n+1}}{\p {x_{j}}}=0,\;\; j=1,\ldots,n-2,
\\& \frac{\p e_{n+1}}{\p {x_{r}}}=-\frac{\sqrt{b^{2}-a^{2}-1}\,\sech\, x_{1}\cdots\, \sech\, x_{{n-3}}}{a \sinh x_{n-2}+b \cosh x_{n-2}} L_{x_{r}},\;\; r=n-1,n.\end{aligned}\end{equation}
After solving system \e{6.24}-\e{6.25}, we get 
\begin{equation}\begin{aligned}&\hskip.0in   L(x_{1},\ldots,x_{n})=c_{1}\sinh x_{1}+\cosh x_{1}\Big\{
c_{2}\sinh x_{2}+c_{3}\sinh x_{3}\cosh x_{2} 
\\& \label{6.26}\hskip.7in +\cdots+c_{n-2}\sinh x_{n-2}\prod_{j=2}^{n-3}\cosh x_{j}  +c_{n-1}\prod_{j=2}^{n-2}\cosh x_{j}\Bigg\}
\\& \hskip.2in +(a \sinh x_{n-2}\!+\!b \cosh x_{n-2})\Bigg(\prod_{j=1}^{n-3} \!\cosh x_{j}\Bigg)\big\{c_{n}\sinh x_{n-1} \\& \hskip.5in +c_{n+1} \cosh x_{n-1}\sinh x_{n}\!+\!c_{n+2} \cosh x_{n-1} \cosh x_{n}\big\}.
\end{aligned}\end{equation}
 Consequently, we conclude  from \e{6.22} and \e{6.26} that  $L$ is congruent to  immersion  (C).
\vskip.05in

If $M^{n}$ is non-minimal and ideal in $H^{m}(-1)$ with type number $\leq 2$, then it follows again from Theorem A, Codazzi's equation and Erbarcher's reduction theorem that  $M^{n}$ is immersed in a totally geodesic  $H^{n+1}(-1)\subset H^{m}(-1)$ for both cases (2) and (3) as well. Therefore, we obtain the same conclusion as above.
\end{proof}

\begin{remark} Minimal ideal hypersurfaces of $S^{n+1}(1)$ with type number $\leq 2$ are either totally geodesic in $S^{n+1}(1)$ or they are  given by case (2) of \cite[Theorem 2]{CY}. Similarly, minimal ideal hypersurfaces of $H^{n+1}(-1)$ with type number $\leq 2$ are either totally geodesic in $H^{n+1}(-1)$ or they are given by case (3) of \cite[Theorem 3]{CY} (cf. \cite[page 423]{book}). Moreover, by using Codazzi's equation, it is easy to verify that each minimal ideal submanifold with type number $\leq 2$ in $S^{m}(1)$ (respectively, in $H^{m}(-1)$) is contained in a totally geodesic $S^{n+1}(1)\subset S^{m}(1)$ (respectively,  in a totally geodesic $H^{n+1}(-1)\subset H^{m}(-1)$).
\end{remark}

\end{document}